\documentclass{birkjour}
 \newtheorem{thm}{Theorem}[section]
 \newtheorem{cor}[thm]{Corollary}
 \newtheorem{lem}[thm]{Lemma}
 \newtheorem{prop}[thm]{Proposition}
 \theoremstyle{definition}
 \newtheorem{defn}[thm]{Definition}
 \newtheorem{ex}[thm]{Example}
 \theoremstyle{remark}
 \newtheorem{rem}[thm]{Remark}
 \numberwithin{equation}{section}
\usepackage{color}
\begin{document}

%
%
%
%
%
%
%
%
%

\title[$K$-Fusion Frames]
 {Characterization and Construction of $K$-Fusion Frames and Their Duals  in Hilbert Spaces}



\author[F. Arabyani Neyshaburi]{Fahimeh Arabyani Neyshaburi}
\address{Department of Mathematics and Computer Sciences, Hakim Sabzevari University, Sabzevar, Iran.}
\email{arabianif@yahoo.com}

\author[A.  Arefijamaal]{Ali Akbar Arefijamaal}
\address{Department of Mathematics and Computer Sciences, Hakim Sabzevari University, Sabzevar, Iran.}
\email{arefijamaal@hsu.ac.ir}

\subjclass{Primary 42C15; Secondary 42C40, 41A58.}

\vspace{1.7cm}
\begin{abstract}
$K$-frames, a new generalization of frames, were recently considered  by L. G$\breve{\textrm{a}}$vru\c{t}a  in connection with atomic systems and some problems arising in sampling theory. Also, fusion frames are an important  generalization of frames, applied  in  a  variety of applications.
In the present paper, we introduce the notion of $K$-fusion frames in Hilbert spaces and obtain several approaches for identifying of $K$-fusion frames. The main  purpose is to reconstruct the elements from the range of the bounded operator $K$ on  a  Hilbert space $\mathcal{H}$  by using a family of closed subspaces in $\mathcal{H}$. This work will  be useful in some problems in sampling theory which are processed by fusion frames. For this end, we present some descriptions for duality of $K$-fusion frames and also resolution of the operator $K$ to provide simple and concrete constructions of duals of $K$-fusion frames. Finally, we survey the robustness of $K$-fusion frames under some perturbations.
\end{abstract}

\maketitle
\textbf{Key words:} Fusion frames; $K$-fusion frames; $K$-duals; resolution of bounded operators.
\maketitle
\section{Introduction and preliminaries}

\smallskip
\goodbreak
Frame theory presents  efficient algorithms for a wide range of applications \cite{Ar13, Ben06, Bod05, Bol98, Cas00}.
In most of  those applications, we deal with  dual frames to reconstruct the modified data and compare it with the original data. In contrast to frames, a new approach so called atomic decomposition for a closed subspace $\mathcal{H}_{0}$ of a Hilbert space $\mathcal{H}$ introduced by  Feichtinger et al.  in \cite{Han} 
 with frame-like properties. However, the sequences in  atomic decompositions
do not necessarily  belong to $\mathcal{H}_{0}$,  this striking property is valuable especially   in sampling theory   \cite{Paw, Werthrt}. Then $K$-frames were introduced  to study atomic systems with respect to a bounded operator $K\in B(\mathcal{H})$ \cite{Gav07}. Indeed, $K$-frames   are equivalent with atomic systems for the operator $K$  and  help us to reconstruct elements from the range of a bounded linear operator $K$ in a separable Hilbert space.
More precisely, let  $\mathcal{H}$ be a separable Hilbert space and $I$  a countable index set, a sequence $F:=\lbrace f_{i}\rbrace_{i \in I} \subseteq \mathcal{H}$ is called a $K$-frame for $\mathcal{H}$, if there exist constants $A, B > 0$ such that
\begin{eqnarray}\label{001}
A \Vert K^{*}f\Vert^{2} \leq \sum_{i\in I} \vert \langle f,f_{i}\rangle\vert^{2} \leq B \Vert f\Vert^{2}, \quad (f\in \mathcal{H}).
\end{eqnarray}
Clearly, if $K=I_{\mathcal{H}}$, then $F$ is an ordinary frame and so $K$-frames arise  as a generalization of the ordinary frames  \cite{Cas00, Chr08, Gav07}. The constants $A$ and $B$ in $(\ref{001})$ are called the lower and the upper bounds of $F$, respectively. Similar to ordinary frames the synthesis operator can be defined as $T_{F}: l^{2}\rightarrow \mathcal{H}$; $T_{F}(\{ c_{i}\}_{i\in I}) = \sum_{i\in I} c_{i}f_{i}$. It is a bounded operator and its adjoint which is called the analysis operator given by $T_{F}^{*}(f)= \{ \langle f,f_{i}\rangle\}_{i\in I}$, and the frame operator is given by $S_{F}: \mathcal{H} \rightarrow \mathcal{H}$; $S_{F}f = T_{F}T_{F}^{*}f = \sum_{i\in I}\langle f,f_{i}\rangle f_{i}$. Unlike ordinary frames, the frame operator of a K-frame is not invertible in general. However,  if $K$ has close range then $S_{F}$ from $R(K)$ onto $S_{F}(R(K))$ is an invertible operator \cite{Xiao}.

The authors in \cite{arefi3} introduced the notion of duality for $K$-frames and presented some methods for construction and characterization of $K$-frames and their duals. Indeed, a Bessel sequence
 $\{g_{i}\}_{i \in I}\subseteq \mathcal{H}$ is called a \textit{$K$-dual} of $\{ f_{i} \}_{i\in I}$ if
\begin{eqnarray}\label{dual1}
Kf = \sum_{i\in I} \langle f,g_{i}\rangle f_{i}, \quad (f\in \mathcal{H}).
\end{eqnarray}
 For further information in $K$-frame theory  we refer the reader  to  \cite{arefi3, Han, Gav07,  Xiao}. The following result  is useful for the proof of  our main results.


\begin{thm}[Douglas   \cite{Douglas}]\label{equ0}
Let $L_{1}\in B(\mathcal{H}_{1}, \mathcal{H})$ and $L_{2}\in B(\mathcal{H}_{2}, \mathcal{H})$ be bounded linear mappings on given Hilbert spaces. Then the following assertions are equivalent:
\begin{itemize}
\item[(i)]
$R(L_{1}) \subseteq R(L_{2})$;
\item[(ii)]
$L_{1}L_{1}^{*} \leq \lambda^{2}L_{2}L_{2}^{*}$, \quad for some $\lambda > 0$;
\item[(iii)]
There exists a bounded linear mapping $X\in L(\mathcal{H}, \mathcal{H}_{2})$, such that $L_{1} = L_{2}X$.
\end{itemize}
Moreover, if (i), (ii) and (iii) are valid, then there exists a unique operator $X$ so that 
\begin{itemize}
\item[(a)]
$\Vert X\Vert ^{2} = \inf \{\alpha>0, L_{1}L_{1}^{*}\leq \alpha L_{2}L_{2}^{*}\}$;
\item[(b)]
$N( L_{1}) = N(X)$;
\item[(c)]
$R(X) \subset \overline{R( L_{2}^{*})} $.
\end{itemize}

\end{thm}

Fusion frame theory is a fundamental mathematical theory   introduced in  \cite{Cas04} to model sensor networks perfectly. 
Although, recent studies shows that fusion frames provide effective frameworks  not only for modeling of sensor networks but also   for signal and image processing, sampling theory, filter banks and a variety of applications that cannot be modeled by discrete frames \cite{Cas08, sensor, hear}. In the following, we review basic definitions and  results of fusion frames.

 Let $\{W_i\}_{i\in I}$ be a family of closed subspaces of $\mathcal{H}$ and $\{\omega_i\}_{i\in I}$ a family of
weights, i.e. $\omega_i>0$, $i\in I$. Then $\{(W_i,\omega_i)\}_{i\in
I}$ is called a \textit{fusion frame} for $\mathcal{H}$ if there exist the
constants $0<A\leq B<\infty$ such that
\begin{eqnarray}\label{Def. fusion}
A\|f\|^{2}\leq \sum_{i\in I}\omega_i^2\|\pi_{W_i}f\|^2\leq
B\|f\|^{2},\qquad (f\in \mathcal{H}),
\end{eqnarray}
where $\pi_{W_{i}}$ denotes the orthogonal projection from Hilbert space $\mathcal{H}$ onto a closed subspace $W_{i}$. The constants $A$ and $B$ are called the \textit{fusion frame
bounds}. If $\omega_i=1$, for all $i$, $W$ is called uniform fusion frame and we denote it by $\lbrace
W_{i}\rbrace_{i\in I}$. Also, if we only have the upper bound in (\ref{Def. fusion}) we
call $\{(W_i,\omega_i)\}_{i\in I}$ a \textit{Bessel fusion
sequence}. 
Recall that for each sequence $\{W_i\}_{i\in I}$ of closed subspaces
in $\mathcal{H}$, the space
\begin{eqnarray*}
\sum_{i\in I}\oplus W_{i} =\left\{\{f_i\}_{i\in I}:f_i\in W_i,
\sum_{i\in I}\|f_i\|^2<\infty\right\},
\end{eqnarray*}
 with the inner product $\left\langle \{f_i\}_{i\in I},\{g_i\}_{i\in I} \right\rangle=\sum_{i\in
I}\langle f_i,g_i \rangle$ is a Hilbert space.

For a Bessel fusion sequence $W := \{(W_i,\omega_i)\}_{i\in I}$ of
$\mathcal{H}$, the \textit{synthesis operator} $T_{W}: \sum_{i\in
I}\oplus W_{i} \rightarrow\mathcal{H}$ is defined by
\begin{equation*}
T_{W}\left(\{f_i\}_{i\in I}\right)=\sum_{i\in I}\omega_if_i,\qquad
\left(\{f_{i}\}_{i\in I}\in \sum_{i\in I}\oplus W_{i}\right).
\end{equation*}
Its adjoint operator $T_{W}^{*}: \mathcal{H}\rightarrow \sum_{i\in
I}\oplus W_{i}$, which is called the \textit{analysis
operator}, is given by
\begin{eqnarray*}
T_{W}^{*}(f)=\left\{\omega_{i}\pi_{W_{i}}(f)\right\}_{i\in I},\qquad (f\in
\mathcal{H}).\end{eqnarray*}
and the \textit{fusion frame operator}
$S_{W}:\mathcal{H}\rightarrow\mathcal{H}$ is defined by $S_{W
}f=\sum_{i\in I}\omega_i^{2}\pi_{W_i}f$, which is a bounded,
invertible and positive operator \cite{Cas04}.

There are some approaches towards
 dual fusion frames, the
first definition  was presented by P.
G$\breve{\textrm{a}}$vru\c{t}a in \cite{Gav02}.
A Bessel fusion sequence $\{(V_i,\nu_i)\}_{i\in I}$ is
called a \textit{dual} fusion frame of $\{(W_i,\omega_i)\}_{i\in I}$ if
\begin{eqnarray}\label{Def:alt}
f=\sum_{i\in I}\omega_{i}\nu_{i}\pi_{V_i}S_{W}^{-1}\pi_{W_i}f,\qquad
(f\in \mathcal{H}).
\end{eqnarray}
The family $\{(S_{W}^{-1}W_i,\omega_i)\}_{i\in I}$, which is also a
fusion frame, is called the \textit{canonical dual} of
$\{(W_i,\omega_i)\}_{i\in I}$. A general approach to dual fusion frames can be found in \cite{Hei15, Hei14}.

Throughout this paper, we suppose  $\mathcal{H}$ is a separable Hilbert space, $K^{\dag}$ the pseudo inverse of operator $K$, $I$ a countable index set and $I_{\mathcal{H}}$ is the identity operator on $\mathcal{H}$. For two Hilbert spaces $\mathcal{H}_{1}$ and $\mathcal{H}_{2}$ we denote by $B(\mathcal{H}_{1},\mathcal{H}_{2})$ the collection of all bounded linear operators between $\mathcal{H}_{1}$ and $\mathcal{H}_{2}$, and we abbreviate $B(\mathcal{H},\mathcal{H})$ by $B(\mathcal{H})$. Also we denote the range of $K\in B(\mathcal{H})$ by $R(K)$, the null space of $K$ by $N(K)$ and the orthogonal projection of $\mathcal{H}$ onto a closed subspace $V \subseteq \mathcal{H}$  by $\pi_{V}$.

The paper is organized as follows. In
Section 2, we describe the notion of $K$-fusion frames and present several methods for  identifying and constructing of $K$-fusion frames.
 Section 3 deals with the duality of $K$-fusion frames, in this section, we introduce  the notion of  dual for $K$-fusion frames and then  we show that in case $K=I_{\mathcal{H}}$ this definition  coincides  with the concept of dual fusion frames, however there are several essentially differences, which we will discuss.  Also, we  present some characterizations for duals of  $K$-fusion frames. Section 4 is devoted to introduce the concept of resolution of a bounded linear operator $K\in B(\mathcal{H})$.  By applying this notion we  obtain more reconstructions from the elements of  $R(K)$.  Finally, we survey the  robustness of $K$-fusion frames and their duals under some perturbations,  in Section 5.

\smallskip
\goodbreak
\section{$K$-fusion frames}

\smallskip
\goodbreak

In this section, we introduce the notion of $K$-fusion frames
in Hilbert spaces and discuss on some their properties. In particular, we present some
approaches for identifying and constructing of $K$-fusion frames. Let us start our consideration with formal definition of $K$-fusion frames.
\begin{defn}\label{kfusion5} Let $\lbrace W_{i}\rbrace_{i\in I}$ be a family
of closed subspaces of $\mathcal{H}$ and $\{\omega_i\}_{i\in I}$   a family
of weights, i.e. $\omega_i>0$, $i\in I$. We call
$W=\lbrace(W_{i},\omega_{i})\rbrace_{i\in I}$ a \textit{$K$-fusion frame} for
$\mathcal{H}$, if there exist positive constants $0 < A, B < \infty$ such
that
\begin{eqnarray}\label{kfusion} A\|K^{*}f\|^{2}\leq \sum_{i\in
I}\omega_i^2\|\pi_{W_i}f\|^2\leq B\|f\|^{2},\qquad (f\in \mathcal{H}).
\end{eqnarray} \end{defn}
The constants $A$ and $B$ in (\ref{kfusion}) are called  lower and
upper bounds of $W$, respectively.
We call $W$ a \textit{minimal $K$-fusion frame}, whenever $W_{i}\cap \overline{span}_{j\in I, j\neq i}W_{j} = \{0\}$  and it is called \textit{exact}, if  for every $j\in I$ the sequence $\lbrace (W_{i},\omega_{i})\rbrace_{i\in I, i\neq j}$ is not a $K$- fusion frame for $\mathcal{H}$.
Obviously, a $K$-fusion frame is a Bessel fusion sequence
and so the synthesis operator, the analysis operator and the frame operator
of $W$ are defined similar to fusion frames, however  for a  $K$-fusion frame, the synthesis operator  is not onto and the frame
operator is not invertible, in general. Furthermore, there are several other differences between fusion frames and $K$-fusion frames. Indeed,  the closed linear span of  $W_{i}$'s which contains $R(K)$ by Theorem \ref{equ0}, is not equal to $\mathcal{H}$.  Also, the following example  shows that,  unlike fusion frames, a minimal $K$-fusion frame  is not necessarily
required to be exact. Take $\mathcal{H} = \mathbb{R}^{4}$ with the orthonormal basis $\{e_{i}\}_{i=1}^{4}$ and
\begin{eqnarray*}
W_{1} = span\{e_{1},e_{2}\} , \quad W_{2} = span\{e_{3}\}.
 \end{eqnarray*}
Define $K\in B(\mathcal{H})$ as
\begin{eqnarray*}
Ke_{1} = e_{1}, \quad Ke_{2} = e_{1}, \quad Ke_{3} = e_{2}.
 \end{eqnarray*}
Then $W=\{(W_{i}, 1)\}_{i=1}^{2}$ is a minimal $K$-fusion frame with the bounds $1/2$ and  $1$. However,  it is not exact  since $\{(W_{1}, 1)\}$ is also a $K$-fusion frame with the same  bounds. In this paper, we will recognize  more differences and similarities of $K$-fusion frames with fusion frames.

\begin{prop}\label{2.3.}
Suppose that $\lbrace(W_{i},\omega_{i})\rbrace_{i\in I}$ is a Bessel fusion sequence
and $K\in B(\mathcal{H})$ is  a closed range operator. The following statements are equivalent.
\item[(i)] The sequence
 $W=\{(W_{i}, \omega_{i})\}_{i\in I}$  is a $K$-fusion frame for $\mathcal{H}$.
\item[(ii)]
 There exists a positive number  $A$  such that $S_{W}\geq AKK^{*}$.
\end{prop}
\begin{proof}
Since $W$ is a Bessel fusion sequence, so it  is a $K$-fusion frame for $\mathcal{H}$  if and only if there exists $A>0$ such that
 \begin{eqnarray*}
A  \langle KK^{*}f, f\rangle  \leq\sum_{i\in I}\omega_{i}^{2}\Vert \pi_{W_{i}}f \Vert^{2}= \langle S_{W}f, f\rangle,
 \end{eqnarray*}
 for every $f\in \mathcal{H}$, or  equivalently $S_{W}\geq AKK^{*}$.
\end{proof}
Notice that,  if $F$ is a Bessel sequence, unlike frames and fusion frames, invertibility of  the frame operator $S_{F}: R(K)\rightarrow S_{F}(R(K))$ does not imply that $F$ is a $K$-frame. For a simple counterexample,   let $K$  be the orthogonal projection onto the subspace generated by $(\frac{1}{\sqrt{2}}, \frac{1}{\sqrt{2}})$. Then $F=\{(1,0)\}$ is not a $K$-frame for  $\mathbb{C}^{2}$, however the operator  $S_{F}:R(K)\rightarrow S_{F}(R(K))$ is invertible.

The following lemmas are necessary for our results.
\begin{lem}\cite{Cas08}\label{2.9}
Let $V$ be a closed subspace of $\mathcal{H}$ and $T$ be a bounded operator
on $\mathcal{H}$. Then \begin{eqnarray}
 \pi_{V}T^{*} = \pi_{V}T^{*}\pi_{\overline{TV}}.
 \end{eqnarray}
\end{lem}
The next  lemma  was shown for fusion frames in \cite{Ruiz}, although we prove it by a  simple  method.
\begin{lem}\label{invert4.}
Let $T\in B(\mathcal{H}_{1}, \mathcal{H}_{2})$ be an invertible operator and  $W=\lbrace (W_{i}, \omega_{i})\rbrace_{i\in I}$ be a Bessel fusion sequence of $\mathcal{H}_{1}$. Then   $\lbrace (TW_{i}, \omega_{i})\rbrace_{i\in I}$ is  a Bessel fusion sequence of $\mathcal{H}_{2}$.
\end{lem}
\begin{proof}
By applying  Lemma \ref{2.9} and the fact that $T$ is invertible, we obtain
\begin{eqnarray*}  \sum_{i\in I}\omega_{i}^{2}\Vert \pi_{TW_{i}}f \Vert^{2}&=&  \sum_{i\in I}\omega_{i}^{2}\Vert \pi_{TW_{i}}(T^{-1})^{*}T^{*}f\Vert^{2} \\
&=&\sum_{i\in I}\omega_{i}^{2}\Vert  \pi_{TW_{i}}(T^{-1})^{*}\pi_{W_{i}}T^{*}f\Vert^{2} \\
&\leq& \Vert T^{-1}\Vert^{2}\sum_{i\in I}\omega_{i}^{2}\Vert \pi_{W_{i}}T^{*}f \Vert^{2}\\
&\leq& B\Vert T^{-1}\Vert^{2} \Vert T\Vert^{2}\Vert f \Vert^{2},
 \end{eqnarray*}
for each $f\in \mathcal{H}_{2}$, as required.
\end{proof}


\begin{thm}\label{4part}
Let $K\in B(\mathcal{H})$ be a closed range operator and  $W=\lbrace (W_{i}, \omega_{i})\rbrace_{i\in I}$  a $K$-fusion frame for
$\mathcal{H}$ with bounds $A$ and $B$, respectively. Then
\begin{itemize}
\item[(i)] If $\lbrace\pi_{R(K)}W_{i},\omega_{i})\rbrace_{i\in I}$ is a Bessel fusion sequence, then $\lbrace (K^{\dag}W_{i},
    \omega_{i})\rbrace_{i\in I}$ is a fusion
    frame for $R(K^{*})$.
\item[(ii)] If $\lbrace\pi_{S_{W}(R(K))}W_{i},\omega_{i})\rbrace_{i\in I}$ is a Bessel fusion  sequence, then  $\lbrace
    (S_{W}^{-1}\pi_{S_{W}(R(K))}W_{i},
    \omega_{i})\rbrace_{i\in I}$
    is a
    fusion
    frame for $R(K)$.
    \item[(iii)]
    If $Q\in B(\mathcal{H})$ is an invertible operator, then $\lbrace (QW_{i}, \omega_{i})\rbrace_{i\in I}$ is a $QK$-fusion frame for $\mathcal{H}$.
  \item[(iv)]
    If $Q\in B(\mathcal{H})$  is an invertible operator and $KQ=QK$. Then $QW$ is also a $K$-fusion frame.
    \item[(v)]
    If $Q\in B(\mathcal{H})$ such that $R(Q)\subseteq R(K)$, then $W$ is also a $Q$-fusion frame.
\end{itemize}
\end{thm}
\begin{proof}
To show $\lbrace (K^{\dag}W_{i},
    \omega_{i})\rbrace_{i\in I}$ is a fusion frame for $R(K^{*})$,  let $f\in R(K^{*})$. Then, we can write
\begin{eqnarray*}  \sum_{i\in I}\omega_{i}^{2}\Vert \pi_{K^{\dag}W_{i}}f \Vert^{2}&=&  \sum_{i\in I}\omega_{i}^{2}\Vert \pi_{K^{\dag}W_{i}}K^{*}(K^{\dag})^{*}f\Vert^{2} \\
&=&\sum_{i\in I}\omega_{i}^{2}\Vert \pi_{K^{\dag}W_{i}}K^{*}\pi_{KK^{\dag}W_{i}}(K^{\dag})^{*}f\Vert^{2} \\
&\leq& \Vert K\Vert^{2}\sum_{i\in I}\omega_{i}^{2}\Vert \pi_{\pi_{R(K)}W_{i}}(K^{\dag})^{*}f \Vert^{2}.
 \end{eqnarray*}
Hence by assumption $\lbrace (K^{\dag}W_{i},
    \omega_{i})\rbrace_{i\in I}$ is a  Bessel fusion  sequence.
On the other hand, there exists
 $A>0$ such that
\begin{eqnarray*} A \|f\|^{2} &=& A\|K^{*}(K^{*})^{\dag}f\|^{2}\\
&\leq& \sum_{i\in I}\omega_{i}^{2}\|\pi_{W_{i}}(K^{*})^{\dag}f\|^{2}\\
&=& \sum_{i\in I}\omega_{i}^{2}\|\pi_{W_{i}}(K^{*})^{\dag}\pi_{K^{\dag}W_{i}}f\|^{2}\\
&\leq&
\|K^{\dag}\|^{2}\sum_{i\in I}\omega_{i}^{2}\|\pi_{K^{\dag}W_{i}}f\|^{2}.
 \end{eqnarray*}
 Therefore, $(i)$  holds. Since the sequence $\{\omega_{i}\pi_{W_{i}}e_{j}\}_{i\in I, j\in J}$ is a $K$-frame for $\mathcal{H}$ the sequence
 $\{(K^{*})^{\dag}K^{*}S_{W}^{-1}\pi_{S_{W}(R(K))}\omega_{i}\pi_{W_{i}}e_{j}\}_{i\in I, j\in J}$
  is a frame for $R(K)$,   by  Corollary 1 in \cite{arefi3}. Hence, there exists $A>0$ such that
\begin{eqnarray*} A \|f\|^{2} &\leq& \sum_{i\in I, j\in J}\mid \langle
f,(K^{*})^{\dag}K^{*}S_{W}^{-1}\pi_{S_{W}(R(K))}\omega_{i}\pi_{W_{i}}e_{j} \rangle\mid^{2}\\
&=& \sum_{i\in I, j\in J}\omega_{i}^{2}\mid \langle
\pi_{W_{i}}\pi_{S_{W}(R(K))}(S_{W}^{-1})^{*}f,e_{j} \rangle\mid^{2}\\
&=& \sum_{i\in I}\omega_{i}^{2}\|\pi_{W_{i}}(S_{W}^{-1}\pi_{S_{W}(R(K))})^{*}f\|^{2}\\
&=&
\sum_{i\in I}\omega_{i}^{2}\|\pi_{W_{i}}(S_{W}^{-1}\pi_{S_{W}(R(K))})^{*}\pi_{S_{W}^{-1}\pi_{S_{W}(R(K))}W_{i}}f\|^{2}\\
&\leq&
\|S_{W}^{-1}\pi_{S_{W}(R(K))}\|^{2}\sum_{i\in I}\omega_{i}^{2}\|\pi_{S_{W}^{-1}\pi_{S_{W}(R(K))}W_{i}}f\|^{2},
 \end{eqnarray*}
 for each $f\in R(K)$. Now, since $S_{W}^{-1}: S_{W}(R(K)) \rightarrow R(K)$  is an invertible operator so
 $\lbrace (S_{W}^{-1}\pi_{S_{W}(R(K))}W_{i},
 \omega_{i})\rbrace_{i\in I}$
is a Bessel fusion sequence for $R(K)$ by Lemma \ref{invert4.}, this follows $(ii)$. To show $(iii)$, suppose that  $Q\in B(\mathcal{H})$ is an invertible operator, then $\lbrace (QW_{i}, \omega_{i})\rbrace_{i\in I}$ is a  Bessel sequence, by Theorem 2.4 in \cite{Gav02}. Moreover,
 \begin{eqnarray*}
  A \|K^{*}Q^{*}f\|^{2} &\leq& \sum_{i\in I}\omega_{i}^{2}\|\pi_{W_{i}}Q^{*}f\|^{2}\\
  &=& \sum_{i\in I}\omega_{i}^{2}\|\pi_{W_{i}}Q^{*}\pi_{QW_{i}}f\|^{2}\\
   &\leq& \|Q\|^{2} \sum_{i\in I}\omega_{i}^{2}\|\pi_{QW_{i}}f\|^{2},
 \end{eqnarray*}
    for every $f\in \mathcal{H}$. Thus, $\lbrace (QW_{i}, \omega_{i})\rbrace_{i\in I}$ is a $QK$-fusion frame. The part $(iv)$ is obtained  by $(iii)$.  Finally, for $(v)$, we have $R(Q)\subseteq R(K)$ so there exists $\lambda >0$ such that $QQ^{*} \leq \lambda^{2}KK^{*}$  by Proposition \ref{equ0}. This follows that
    \begin{eqnarray*}
    \frac{A}{\lambda^{2}}\|Q^{*}f\|^{2} \leq \sum_{i\in I}\omega_{i}^{2}\|\pi_{W_{i}}f\|^{2},
     \end{eqnarray*}
Therefore  $W$ is a $Q$-fusion frame for $\mathcal{H}$.
\end{proof}
Notice that, the condition in  Theorem \ref{4part} (ii) is established in many statuses. For example, if  for all $i\in I$  either $W_{i}\subseteq S_{W}(R(K))$ or $\pi_{S_{W}(R(K))}W_{i}\subseteq W_{i}$. Also,
  Theorem \ref{4part} $(v)$ is a generalization of  Proposition 3.3 in \cite{sema}. In the end of this section, we present the second  approach for constructing of  $K$-fusion frames.
\begin{thm}\label{const,K-fusion}
Let  $K$ be a closed range operator and $W = \lbrace (W_{i}, \omega_{i})\rbrace_{i\in I}$
a  fusion frame for $R(K^{*})$.
 Then $\lbrace (KW_{i}, \omega_{i})\rbrace_{i\in
I}$ is a $K$-fusion frame for $\mathcal{H}$.
\end{thm}
\begin{proof}
It is not difficult to see that every Bessel fusion sequence for a closed subspace of $\mathcal{H}$ is also a Bessel fusion sequence for
$\mathcal{H}$. Suppose  $B$ is an upper bound for  $W$ as  a Bessel fusion sequence for
$\mathcal{H}$. Also, let $f\in \mathcal{H}$ we can write $f = g+h$ which $g\in R(K)$ and $h\in (R(K))^{\perp}$. Thus
\begin{eqnarray*}
\sum_{i\in I}\omega_i^2 \|\pi_{KW_i}f\|^2
&=&\sum_{i\in I}\omega_i^2 \|\pi_{KW_i}g\|^2\\
 &=& \sum_{i\in I}\omega_i^2 \|\pi_{KW_i}(K^{\dag})^{*}K^{*}g\|^2\\
&=& \sum_{i\in I}\omega_i^2 \|\pi_{KW_i}(K^{\dag})^{*}\pi_{K^{\dag}KW_{i}}K^{*}g\|^2\\
&\leq& \Vert K^{\dag}\Vert^{2} \sum_{i\in I}\omega_i^2 \|\pi_{\pi_{R(K^{\dag})}W_{i}}K^{*}g\|^2\\
&\leq&  \Vert K^{\dag}\Vert^{2} \sum_{i\in I}\omega_i^2 \|\pi_{W_i}K^{*}g\|^2\\
&\leq&    \Vert K\Vert^{2}  \Vert K^{\dag}\Vert^{2} B \Vert f\Vert^{2}.
\end{eqnarray*}
 Hence $\lbrace (KW_{i}, \omega_{i})\rbrace_{i\in
I}$ is a Bessel fusion sequence.
Moreover, there exists $A>0$ such that
\begin{eqnarray*}
A \|K^{*}f\|^{2} &\leq& \sum_{i\in I}\omega_i^2\|\pi_{W_i}K^{*}f\|^2\\
 &=&\sum_{i\in I}\omega_i^2\|\pi_{W_i}K^{*}\pi_{KW_i}f\|^2 \\
 &\leq& \|K\|^{2}\sum_{i\in I}\omega_i^2\|\pi_{KW_i}f\|^2.
 \end{eqnarray*}
 This follows the result.
\end{proof}
 By applying Theorem \ref{const,K-fusion} the following result immediately  is obtained.
\begin{cor}
Let  $K$ be a closed range operator and $\lbrace (W_{i}, \omega_{i})\rbrace_{i\in I}$
a  fusion frame for $\mathcal{H}$.
 Then $\lbrace (K(W_{i}\cap R(K^{*})), \omega_{i})\rbrace_{i\in
I}$ is a $K$-fusion frame for $\mathcal{H}$.
\end{cor}
\section{Duality of  $K$-fusion frames}

\smallskip
\goodbreak
In this section, we present some descriptions for duality of $K$-fusion frames. Then, we try to characterize and identify duals of $K$-fusion frames. Our approach to define the duality of $K$-fusion frames is a generalization of the  idea in \cite{Hei14}.

\begin{defn}
Let  $W = \lbrace (W_{i}, \omega_{i})\rbrace_{i\in I}$   be a  $K$-fusion frame, 
 a Bessel fusion sequence $V=\lbrace(V_{i},\upsilon_{i})\rbrace_{i\in I}$ is called a  $QK$-dual of $W$ if there exists a bounded linear operator $Q:\sum_{i\in I}\oplus W_{i}\rightarrow \sum_{i\in I}\oplus V_{i}$ such that
\begin{eqnarray}\label{QK dual}
T_{W}Q^{*}T^{*}_{V} = K.
\end{eqnarray}
\end{defn}

Every   $QK$-dual of $W$
 is a $K^{*}$-fusion frame.
More precisely, if  $V=\lbrace(V_{i},\upsilon_{i})\rbrace_{i\in I}$ is  a $QK$-dual of $W$, we can write
\begin{eqnarray*}
\|Kf\|^{4} &=& |\langle Kf, Kf \rangle|^{2}\\
&=&  \left|\left\langle T_{W}Q^{*}T^{*}_{V} f, Kf\right\rangle\right|^{2}\\
&=& \left|\left\langle T^{*}_{V} f, QT_{W}^{*}Kf\right\rangle\right|^{2}\\
&\leq& \Vert T^{*}_{V} f \Vert^{2}  \Vert Q\Vert^{2}B\Vert  Kf \Vert^{2}\\
&=&\Vert Q\Vert^{2}B\Vert  Kf \Vert^{2}\sum_{i\in I}\upsilon_{i}^{2}\Vert \pi_{V_{i}}f\Vert^{2},
\end{eqnarray*}
 for every $f\in \mathcal{H}$, where $B$ is an upper bound of $W$.
 Moreover, if $C$ and $D$ are the optimal bounds of $V$, respectively. Then
\begin{eqnarray*}
C\geq B^{-1} \Vert Q\Vert^{-2} \quad  \textit{and} \quad
D\geq A^{-1} \Vert Q\Vert^{-2},
\end{eqnarray*}
in which $A$ and $B$ are  the optimal bounds of $W$, respectively. 

\begin{rem}
Consider a $K$-fusion frame $W = \lbrace (W_{i}, \omega_{i})\rbrace_{i\in I}$ for $\mathcal{H}$.    Applying the Douglas' theorem \cite{Douglas} there exists an operator $X\in B(\mathcal{H}, \sum_{i\in I}\oplus W_{i})$ such that
\begin{eqnarray}\label{DougX}
 T_{W}X=K.
\end{eqnarray}
 We denote the $i$-th component of $Xf$ by $X_{i}f=(Xf)_{i}$ and clearly $X_{i}\in B(\mathcal{H}, W_{i})$. 
\end{rem}
 In the next theorem, we show that by these operators  one may construct some $QK$-duals for $W$.
\begin{thm}\label{3.188}
Let $W = \lbrace (W_{i}, \omega_{i})\rbrace_{i\in I}$ be a $K$-fusion frame and $X$ be an operator  as in $(\ref{DougX})$. If $\widehat{W}=\{X_{i}^{*}W_{i}\}_{i\in I}$ is a Bessel fusion sequence, then it is a $QK$-dual for $W$.
\end{thm}
\begin{proof}
Define the mapping $\Gamma:R(T_{\widehat{W}}^{*})\rightarrow \sum_{i\in I}\oplus W_{i}$ so that $\Gamma T_{\widehat{W}}^{*}f = Xf$. Then $\Gamma$ is well-defined and bounded. Indeed, for every $f\in \mathcal{H}$ if $\Gamma T_{\widehat{W}}^{*}f = \{\pi_{X_{i}^{*}W_{i}}f\}_{i\in I}=0$ we imply that
\begin{eqnarray*}
f\in (X_{i}^{*}W_{i})^{\perp} = (R(X_{i}^{*}))^{\perp} = N(X_{i}), \quad (i\in I),
\end{eqnarray*}
i.e., $Xf=0$. Moreover, 
\begin{eqnarray*}
\Vert \Gamma \{\pi_{X_{i}^{*}W_{i}}f\}_{i\in I}\Vert^{2}&=& \Vert Xf\Vert^{2}\\
&=& \sum_{i\in I}\Vert \pi_{W_{i}}X_{i}f\Vert^{2}\\
&=&\sum_{i\in I}\Vert \pi_{W_{i}}X_{i}\pi_{X_{i}^{*}W_{i}}f\Vert^{2}\\
&\leq&\Vert X\Vert^{2} \sum_{i\in I}\Vert \pi_{X_{i}^{*}W_{i}}f\Vert^{2}.
\end{eqnarray*}
Hence, $\Gamma$ can be uniquely extended to $\overline{R(T_{\widehat{W}}^{*})}$. Also, we take $\Gamma=0$ on $R(T_{\widehat{W}}^{*})^{\perp}$ and let $Q=\Gamma^{*}$. This implies that $Q^{*}\in B(\sum_{i\in I}\oplus X_{i}^{*}W_{i}, \sum_{i\in I}\oplus W_{i})$ and 
\begin{eqnarray*}
T_{W}Q^{*}T_{\widehat{W}}^{*}= T_{W}X  = K,
\end{eqnarray*}
as required.
\end{proof}
\begin{ex}\label{example002}
Consider $\mathcal{H} = \mathbb{R}^{3}$ and  define $K\in B(\mathcal{H})$
as\begin{eqnarray*}
Ke_{1} = e_{1}+e_{2}, \quad Ke_{2} = e_{3}, \quad Ke_{3} = 0,
\end{eqnarray*}
where $\lbrace e_{i}\rbrace_{i=1}^{3}$ is the standard orthonormal basis of
$\mathcal{H}$. Also let
\begin{eqnarray*}
W_{1} = \textit{span}\lbrace e_{1}+e_{2}, e_{3}\rbrace, \quad W_{2} = \textit{span}\lbrace e_{3}\rbrace, \quad W_{3} = \textit{span}\lbrace e_{1}+e_{2}\rbrace,
\end{eqnarray*}
and $\omega_{i}=1$, for all $1 \leq i\leq 3$. Then $W=\lbrace (W_{i},
\omega_{i})\rbrace_{i=1}^{3}$ is a $K$-fusion frame with bounds $1$ and $2$,
respectively. Now, define  the operator $X:\mathbb{R}^{3}\rightarrow \sum_{i\in 1}^{3}\oplus W_{i}$ as
\begin{eqnarray*}
Xf= \left\{\left(\dfrac{a}{2}, \dfrac{a}{2}, \dfrac{b}{2}\right), \left(0, 0, \dfrac{b}{2}\right), \left(\dfrac{a}{2}, \dfrac{a}{2}, 0\right)   \right\}
\end{eqnarray*}
for every $f=(a,b,c)\in \mathbb{R}^{3}$. One can easily see that $T_{W}X=K$ and
\begin{eqnarray*}
X_{1}^{*}W_{1}=\textit{span}\lbrace e_{1}, e_{2}\rbrace, \quad
X_{2}^{*}W_{2}=  \textit{span}\lbrace e_{2}\rbrace, \quad
X_{3}^{*}W_{3}=  \textit{span}\lbrace e_{1}\rbrace.
\end{eqnarray*}
Hence, $\{X_{i}^{*}W_{i}\}_{i\in I}$ is a  $QK$-dual for $W$. Moreover, $\Vert X\Vert=1$, $\Vert K^{*}f\Vert^{2}=(a+b)^{2}+c^{2}$ and $\Vert T_{W}^{*}f\Vert^{2}=(a+b)^{2}+2c^{2}$ and so 
\begin{eqnarray*}
\inf \{\alpha>0,\quad \Vert K^{*}f\Vert ^{2}\leq \alpha\Vert T_{W}^{*}f\Vert ^{2}\}=1=\Vert X\Vert^{2}.
\end{eqnarray*}
 Also, $N(X)=N(K)$ and $R(X)\subseteq \overline{R(T_{W}^{*})}$. This shows that the operator $X$ is the unique operator, which satisfies all items in Douglas' theorem. 
\end{ex}

Notice that in Theorem \ref{3.188}, $\widehat{W}$ is not necessarily Bessel fusion sequence. In fact, 
a simple computation shows that $X^{*}\{g_{i}\}_{i\in I}= \sum_{i\in I}X_{i}^{*}g_{i}$, for all $\{g_{i}\}_{i\in I}\in \sum_{i\in I}\oplus W_{i}$. So for every $K$-fusion frame such that $W_{i}\perp W_{j}$, for all $i\neq j$, we obtain 
\begin{eqnarray*}
X_{i}^{*}W_{i} = \{X^{*}T_{W}^{*}f_{i}; \quad f_{i}\in W_{i}\} = K^{*}W_{i}, \quad (i\in I).
\end{eqnarray*}
 Now, let  $\lbrace e_{i}\rbrace_{i\in I}$ be an
orthonormal basis of $\mathcal{H}$ and $W_{i} = span\{e_{i}\}$, for all $i\in I$.
Clearly $\lbrace W_{i}\rbrace_{i\in I}$ is an orthonormal fusion basis and also a $K$-fusion frame for $\mathcal{H}$. Define
\begin{eqnarray*}
K^{*}e_{i} =\begin{cases}
\begin{array}{ccc}
\dfrac{1}{m}e_{1}& \;
{i=2m-1}, \\
e_{m}& \; {i=2m}. \\
\end{array}
\end{cases}
\end{eqnarray*}
Then the mapping $K^{*}$ can be extended to a bounded and surjective linear operator on $\mathcal{H}$,
i.e., $K^{*}\in B(\mathcal{H})$. Moreover,
\begin{eqnarray*}
K^{*}W_{i} =\begin{cases}
\begin{array}{ccc}
span\{e_{1}\}& \;
{i=2m-1}, \\
span\{e_{m}\}& \; {i=2m}. \\
\end{array}
\end{cases}
\end{eqnarray*}
Thus,  for $f=e_{1}$
\begin{eqnarray*}
\sum_{i=1}^{n}\pi_{K^{*}W_{i}}f \rightarrow \infty, \quad (n\rightarrow \infty),
\end{eqnarray*}
i.e. $
\lbrace X_{i}^{*}W_{i} \rbrace_{i=1}^{\infty}  =\lbrace K^{*}W_{i} \rbrace_{i=1}^{\infty}$ is not a Bessel fusion sequence.

Using the Douglas' theorem, the equation $T_{W}X=K$ has a unique solution as $X_{w}$ such that 
\begin{eqnarray}\label{XW}
\Vert X_{w}\Vert ^{2} = \inf \{\alpha>0, \Vert K^{*}f\Vert ^{2}\leq \alpha\Vert T_{W}^{*}f\Vert ^{2}; f\in \mathcal{H}\}.
\end{eqnarray} 
It is worth to note that, in case  $K=I_{\mathcal{H}}$ we obtain $X_{w}=T_{W}S_{W}^{-1}$ and so the $QK$-dual $\{X_{i}^{*}W_{i}\}_{i\in I}$ of $W$ is exactly $\{S_{W}^{-1}W_{i}\}$. By these considerations, we can obtain optimal bounds of a $K$-fusion frame.
Let   $W = \lbrace (W_{i}, \omega_{i})\rbrace_{i\in I}$ be $K$-fusion frame with optimal bounds $A$ and $B$, respectively. Then the upper bound is obtained directly by definition as $B=\Vert S_{W}\Vert$.  Also
\begin{eqnarray*}
A &=& sup\{\alpha>0: \quad \alpha \Vert K^{*}f\Vert^{2}\leq \Vert T_{W}^{*}f\Vert^{2}, f\in \mathcal{H}\}\\
&=&\left(inf\{\beta>0: \quad  \Vert K^{*}f\Vert^{2}\leq \beta\Vert T_{W}^{*}f\Vert^{2}, f\in \mathcal{H}\}\right)^{-1}\\
&=& \Vert X_{w}\Vert^{-2}.
\end{eqnarray*} 
As a considerable result, we get the optimal lower bound of  fusion frames.
\begin{cor}
Let   $W = \lbrace (W_{i}, \omega_{i})\rbrace_{i\in I}$ be a fusion frame with the optimal lower bound $A$. Then 
\begin{eqnarray*}
A = \Vert T_{W}^{*}S_{W}^{-1}\Vert^{-2}.
\end{eqnarray*} 
\end{cor}

 Recall that, a bounded operator $Q:\sum_{i\in I}\oplus W_{i}\rightarrow \sum_{i\in I}\oplus V_{i}$  is component preserving \cite{Hei14}, whenever
\begin{equation*}
Qp_{i}\sum_{i\in I}\oplus W_{i} = p_{i}\sum_{i\in I}\oplus V_{i},
\end{equation*}
where
\begin{eqnarray*}
p_{i}\{f_{j}\}_{j\in I} =\begin{cases}
\begin{array}{ccc}
f_{i}& \;
{i=j}, \\
0& \; {i\neq} j, \\
\end{array}
\end{cases}
\end{eqnarray*}
 If the operator $Q$  in (\ref{QK dual}) is component preserving then $V$ is called $QK$-component preserving dual of $W$.
 By a similar argument with  \cite{Hei14} we obtain   the following characterization of $QK$-component preserving duals of  a $K$-fusion frame, so we avoid the burden of proof.
\begin{thm}
Let $W = \lbrace (W_{i}, \omega_{i})\rbrace_{i\in I}$  be a $K$-fusion frame such that $\omega_{i}>\delta>0$, for some $\delta>0$ and $i\in I$. Then a Bessel fusion sequence $V = \lbrace (V_{i}, \upsilon_{i})\rbrace_{i\in I}$ is a  $QK$-component preserving dual of $W$ if and only if $V_{i} = \Psi p_{i}\sum_{j\in I}\oplus W_{j}$, in which $\Psi\in B(\sum_{i\in I}\oplus W_{i}, \mathcal{H})$ such that $\Psi T_{W}^{*} = K^{*}$.
\end{thm}

\subsection{$K$-Duals}
In the squel, we present the other approach to reconstruct   the elements of $R(K)$. To this end, we  generalize  duality introduced by G$\breve{\textrm{a}}$vru\c{t}a in \cite{Gav02}. This approach gives us an explicit form for dual of $K$-fusion frames, which is coincident with the canonical dual of fusion frames in case $K=I_{\mathcal{H}}$. Moreover, we obtain several methods for constructing and characterization of  duals of  $K$-fusion frames.
Let   $W = \lbrace (W_{i}, \omega_{i})\rbrace_{i\in I}$ be a $K$-fusion frame, we can write
\begin{eqnarray*}
Kf &=& S_{W}^{*}(S_{W}^{-1})^{*}Kf\\
&=&\sum_{i\in I}\omega_{i}^{2}\pi_{R(K)}\pi_{W_{i}}(S_{W}^{-1})^{*}Kf\\
&=&\sum_{i\in I}\omega_{i}^{2}\pi_{R(K)}\pi_{W_{i}}(S_{W}^{-1})^{*}K\pi_{K^{*}S_{W}^{-1}\pi_{S_{W}(R(K))}W_{i}}f.
\end{eqnarray*}
Hence,  we obtain  the following definition, which is also a  special status of  (\ref{QK dual}) by taking  
\begin{equation*}
Q^{*}\{f_{i}\}_{i\in I}  =\phi_{vw}\{f_{i}\}_{i\in I} = \{\pi_{W_{i}}(S_{W}^{-1})^{*}Kf_{i}\}_{i\in I}.
\end{equation*}
\begin{defn}\label{10}
Let $W = \lbrace (W_{i}, \omega_{i})\rbrace_{i\in I}$ be a $K$-fusion frame.
A Bessel fusion sequence $\lbrace(V_{i},\upsilon_{i})\rbrace_{i\in I}$ is
called a $K$-dual of $W$ if
\begin{eqnarray}\label{Kdual2}
K = \pi_{R(K)}T_{W}\phi_{vw}T^{*}_{V}.
 \end{eqnarray}
\end{defn}

\begin{rem}
\begin{itemize}
\item[(a)] Let $K = I_{\mathcal{H}}$ in (\ref{Kdual2}), we easily see that $V$ is a dual of $W$ in the notion of \cite{Gav02}.
\item[(b)]If  $\widetilde{W}:=\lbrace
(K^{*}S_{W}^{-1}\pi_{S_{W}(R(K))}W_{i},\omega_{i})\rbrace_{i\in I}$ is a Bessel fusion
sequence, then  it  is a $K$-dual for $W$  and in this case  we call it the canonical $K$-dual of $W$.
\item[(c)]  The sequence $\widetilde{W}$ is not a  Bessel fusion sequence, necessarily. In the following, we  illustrate this fact.
\end{itemize}
\end{rem}
\begin{ex}
Let $\mathcal{H} = l^{2}$  with the  standard orthonormal basis
$\lbrace e_{n}\rbrace_{n=1}^{\infty}$. Define
\begin{eqnarray*}
Ke_{i} =\begin{cases}
\begin{array}{ccc}
\sum_{m=1}^{\infty}\dfrac{1}{m^{2}}e_{2m-1}& \;
{i=1}, \\
0& \; {i=2}, \\
e_{8m}& \; {i=m+2, \quad (m\in \mathbb{N})}. \\
\end{array}
\end{cases}
\end{eqnarray*}
Then $K\in B(\mathcal{H})$ and $K^{*}: \mathcal{H}\rightarrow  \mathcal{H}$ is given by
\begin{eqnarray*}
K^{*}e_{i} =\begin{cases}
\begin{array}{ccc}
\dfrac{1}{m^{2}}e_{1}& \;
{i=2m-1}, \\
e_{m+2}& \; {i=8m}, \\
0& \; {\textit{otherwise}}. \\
\end{array}
\end{cases}
\end{eqnarray*}
Now, take
$W_{i} = \overline{span}\{e_{i}\}$, for all $i\neq 2, 4$,
$W_{2} = W_{4} =  \overline{span}\{e_{2}+e_{4}\}$
and $\omega_{i}=1$, for all $i$.
Then $\lbrace W_{i}\rbrace_{i=1}^{\infty}$ is a $K$-fusion frame. More precisely, for every  $f=\{a_{i}\}_{i=1}^{\infty}\in \mathcal{H}$ we have
\begin{equation*}
\Vert K^{*}f\Vert^{2} \leq \sum_{i=1}^{\infty}\vert a_{8m}\vert^{2}+\sum_{i=1}^{\infty}\vert a_{2m-1}\vert^{2} \leq
\sum_{i=1}^{\infty}\Vert \pi_{W_{i}}f \Vert^{2} \leq 2\Vert f\Vert^{2},
\end{equation*}
Furthermore,
\begin{equation*}
\pi_{R(K)}f = \sum_{m=1}^{\infty}a_{8m} e_{8m} + \dfrac{90}{\pi^{4}}\sum_{m=1}^{\infty} \left( \sum_{j=1}^{\infty} \dfrac{1}{j^{2}}a_{2j-1}\right)\dfrac{1}{m^{2}}e_{2m-1},
\end{equation*}
and
\begin{equation*}
S_{W}f = \left(  a_{1}, a_{2}+a_{4}, a_{3}, a_{2}+a_{4}, a_{5}, a_{2}+a_{4}, a_{6}, ... \right).
\end{equation*}
Therefore,  a direct calculation shows that
\begin{equation*}
K^{*}S_{W}^{-1}\pi_{S_{W}(R(K))}W_{2m-1} = \overline{span}\{e_{1}\},
\end{equation*}
 for all $1 \leq m< \infty$, i.e.,  $\left\{K^{*}S_{W}^{-1}\pi_{S_{W}(R(K))}W_{i}\right\}_{i=1}^{\infty}$ is not a Bessel fusion  sequence.
\end{ex}

It is worth noticing that, when  $\lbrace (\pi_{S_{W}(R(K))}W_{i},\omega_{i})\rbrace_{i\in I}$ is a Bessel fusion sequence
 with a Bessel  bound $B$ then $\widetilde{W} = \lbrace
(K^{*}S_{W}^{-1}\pi_{S_{W}(R(K))}W_{i},\omega_{i})\rbrace_{i\in I}$  is also
a Bessel fusion sequence with the Bessel bound $ B\Vert K\Vert^{2}\Vert K^{\dagger}\Vert^{2}\Vert S_{W}\Vert^{2}\Vert S_{W}^{-1}\Vert^{2}$. To show this, assume  that  $f\in \mathcal{H}$ and $f = g+h$, where $g\in R(K^{*})$ and $h\in (R(K^{*}))^{\perp}$, then
\begin{eqnarray*}
 \sum_{i\in I} \Vert \pi_{\widetilde{W}_{i}}f\Vert^{2}
&=&\sum_{i\in I} \omega_{i}^{2}\Vert \pi_{\widetilde{W}_{i}}g\Vert^{2}\\
&=& \sum_{i\in I} \omega_{i}^{2}\Vert \pi_{\widetilde{W}_{i}}K^{\dagger}Kg\Vert^{2}\\
&=& \sum_{i\in I} \omega_{i}^{2}\Vert \pi_{\widetilde{W}_{i}}K^{\dagger}\pi_{S_{W}^{-1}\pi_{S_{W}(R(K))}W_{i}}Kg\Vert^{2}\\
&\leq&  B\Vert K\Vert^{2}\Vert K^{\dagger}\Vert^{2}\Vert S_{W}\Vert^{2}\Vert S_{W}^{-1}\Vert^{2}\Vert f\Vert^{2},
\end{eqnarray*}
where the last inequality is obtained by Lemma \ref{invert4.}.

Now, we are going to present a simple method for constructing of  $K$-duals by the canonical $K$-dual. For this, let
$W = \lbrace (W_{i}, \omega_{i})\rbrace_{i\in I}$ be a $K$-fusion frame with the canonical $K$-dual $\widetilde{W}$ such that $K^{*}S_{W}^{-1}\pi_{S_{W}(R(K))}W_{j} \neq \mathcal{H}$, for some $j\in I$. This implies that
$(K^{*}S_{W}^{-1}\pi_{S_{W}(R(K))}W_{j})^{\perp} \neq \lbrace0\rbrace$. Take
\begin{eqnarray*}
V_{j} =
K^{*}S_{W}^{-1}\pi_{S_{W}(R(K))}W_{j}\oplus U_{j},
\end{eqnarray*}
  where $U_{j}$ is a closed subspace of $\left(K^{*}S_{W}^{-1}\pi_{S_{W}(R(K))}W_{j}\right)^{\perp}$, and for all $i \neq j$ consider   $V_{i} =
K^{*}S_{W}^{-1}\pi_{S_{W}(R(K))}W_{i}$. Then,  $V=\lbrace(V_{i},\omega_{i})\rbrace_{i\in I}$ is a Bessel fusion sequence
and clearly it
 is a $K$-dual of $\lbrace (W_{i}, \omega_{i})\rbrace_{i\in I}$ different  from the canonical $K$-dual. More precisely,  for every $f\in \mathcal{H}$
\begin{eqnarray*}
&\sum_{i\in I}&\omega_{i}^{2}\pi_{R(K)}\pi_{W_{i}}(S_{W}^{-1})^{*}K\pi_{V_{i}}f\\
&=&
\omega_{j}^{2}\pi_{R(K)}\pi_{W_{j}}(S_{W}^{-1})^{*}K\pi_{V_{j}} f
+ \sum_{i\in I, i\neq j}\omega_{i}^{2}\pi_{R(K)}\pi_{W_{i}}(S_{W}^{-1})^{*}K\pi_{V_{i}}f\\
&=& \omega_{j}^{2}\pi_{R(K)}\pi_{W_{j}}(S_{W}^{-1})^{*}K\pi_{U_{j}} f
+\omega_{j}^{2}\pi_{R(K)}\pi_{W_{j}}(S_{W}^{-1})^{*}K\pi_{K^{*}S_{W}^{-1}\pi_{S_{W}(R(K))}W_{j}}f\\
&+& \sum_{i\in I, i\neq j}\omega_{i}^{2}\pi_{R(K)}\pi_{W_{i}}(S_{W}^{-1})^{*}K\pi_{V_{i}}f\\
&=&\omega_{j}^{2}\pi_{R(K)}\pi_{W_{j}}(S_{W}^{-1})^{*}K\pi_{K^{*}S_{W}^{-1}\pi_{S_{W}(R(K))}W_{j}}\pi_{U_{j}}f\\
&+& \omega_{j}^{2}\pi_{R(K)}\pi_{W_{j}}(S_{W}^{-1})^{*}K\pi_{K^{*}S_{W}^{-1}\pi_{S_{W}(R(K))}W_{j}}f+ \sum_{i\in I, i\neq j}\omega_{i}^{2}\pi_{R(K)}\pi_{W_{i}}(S_{W}^{-1})^{*}K\pi_{V_{i}}f\\
&=& \sum_{i\in I}\omega_{i}^{2}\pi_{R(K)}\pi_{W_{i}}(S_{W}^{-1})^{*}K\pi_{K^{*}S_{W}^{-1}\pi_{S_{W}(R(K))}W_{i}}f=Kf.
\end{eqnarray*}
Now, let us turn to the example.
\begin{ex}\label{example2}
Suppose  $\mathcal{H}$, $K$ and $W$  are as in Example \ref{example002}. Then 
we have
\begin{equation*}
S_{W}\pi_{R(K)} = \left[
 \begin{array}{ccc}

1 \quad 1 \quad 0\\

1 \quad 1 \quad 0\\

0 \quad 0 \quad 2\\

\end{array} \right].
\end{equation*}
Therefore,
\begin{equation*}
S_{W}^{-1}\pi_{S_{W}(R(K))} = \left[
 \begin{array}{ccc}

1/4 \quad 1/4 \quad 0\\

1/4 \quad 1/4 \quad 0\\

0 \quad 0 \quad 1/2\\

\end{array} \right].
\end{equation*}
Now, a straightforward calculation shows that
$S_{W}^{-1}\pi_{S_{W}(R(K))}W_{i} = W_{i}$, for every $1 \leq i\leq 3$. Hence, the canonical $K$-dual $\widetilde{W} = \lbrace
(K^{*}S_{W}^{-1}\pi_{S_{W}(R(K))}W_{i},\omega_{i})\rbrace_{i\in I}$ is obtained as the following

\begin{center}
$K^{*}S_{W}^{-1}\pi_{S_{W}(R(K))} W_{1} =  \textit{span}\lbrace e_{1}, e_{2}\rbrace$, \\
$K^{*}S_{W}^{-1}\pi_{S_{W}(R(K))} W_{2} =  \textit{span}\lbrace e_{2}\rbrace$, \\
$K^{*}S_{W}^{-1}\pi_{S_{W}(R(K))} W_{3} =  \textit{span}\lbrace e_{1}\rbrace$.
\end{center}
Also, consider
\begin{equation*}
V_{1} = K^{*}S_{W}^{-1}\pi_{S_{W}(R(K))} W_{1},\quad V_{2} =K^{*}S_{W}^{-1}\pi_{S_{W}(R(K))} W_{2},\quad V_{3} = \textit{span}\lbrace e_{1}, e_{3}\rbrace.
\end{equation*}
Then, $\{(V_{i}, 1)\}_{i\in I}$ is a $K$-dual of $W$ different from the canonical $K$-dual.
\end{ex}
It is worth to note that, in the above example the canonical $K$-dual is exactly the unique $QK$-dual of Example \ref{example002}. This comes from the fact that, in this $K$-fusion frame $S_{W}(R(K))\subset R(K)$. More general,  we have the following result.
\begin{thm}
Let
 $K$ be a closed range operator and $W = \lbrace (W_{i}, \omega_{i})\rbrace_{i\in I}$ a $K$-fusion frame for $\mathcal{H}$.  Then,  $\widetilde{W}=\{(X_{w}^{*})_{i}W_{i}\}_{i\in I}$  if and only if $S_{W}(S_{W}(R(K)))\subseteq R(K)$.
\end{thm}
\begin{proof}
First, suppose $S_{W}(S_{W}(R(K)))\subseteq R(K)$.  Obviously,  the equation 
\begin{eqnarray}\label{equ1}
\pi_{R(K)}T_{W}X=K,
\end{eqnarray}
has a solution as $M:=T_{W}^{*}(S_{W}^{-1})^{*}K$. Applying the assumption  we obtain $T_{W}M=K$, i.e., $M$ satisfies $(\ref{DougX})$. Also, for every $f\in N(M)$ we obtain 
\begin{eqnarray*}
Kf=\pi_{R(K)}T_{W}T_{W}^{*}(S_{W}^{-1})^{*}K=\pi_{R(K)}T_{W} Mf=0.
\end{eqnarray*}
Hence, $N(M)=N(K)$ and clearly $R(M)\subseteq \overline{R(T_{W}^{*})}$. Thus, 
$M=X_{w}$ by Theorem \ref{equ0}. Also, $M_{i}=\pi_{W_{i}}(S_{W}^{-1})^{*}K$   so
\begin{eqnarray*}
M_{i}^{*}W_{i}=K^{*}S_{W}^{-1}\pi_{S_{W}(R(K))}W_{i}, \quad (i\in I),
\end{eqnarray*}
or equivalently, $\widetilde{W}=\{(X_{w}^{*})_{i}W_{i}\}_{i\in I}$. Conversely, let $\widetilde{W}=\{(X_{w}^{*})_{i}W_{i}\}_{i\in I}$. Then, 
\begin{eqnarray*}
X_{w}^{*}\{f_{i}\}_{i\in I} &=&\sum_{i\in I}(X_{w}^{*})_{i}f_{i}\\
&=&\sum_{i\in I}K^{*}S_{W}^{-1}\pi_{S_{W}(R(K))}f_{i}\\
&=&K^{*}S_{W}^{-1}\pi_{S_{W}(R(K))}T_{W}\{f_{i}\}_{i\in I},
\end{eqnarray*}
for all $\{f_{i}\}_{i\in I}\in \sum_{i\in I}\oplus W_{i}$. Therefore, $X_{w}=T_{W}^{*}(S_{W}^{-1})^{*}K$,   and   the operator $T_{W}^{*}(S_{W}^{-1})^{*}K$ satisfies 
\begin{eqnarray*}
T_{W}T_{W}^{*}(S_{W}^{-1})^{*}K=K=\pi_{R(K)}T_{W}T_{W}^{*}(S_{W}^{-1})^{*}K,
\end{eqnarray*}
i.e., $S_{W}(S_{W}(R(K)))\subseteq R(K)$, as required.
\end{proof}
In the sequel, we   characterizes all $K$-duals of   minimal $K$-fusion frames, under some condition. For this, we need to a simple  lemma, which prove it for convenience.
\begin{lem}\label{prop1R}
Let $K$ be a closed range operator and $F=\{ f_{i} \}_{i\in I}$ be a $K$-frame  for $\mathcal{H}$. Then $\{\pi_{R(K)} f_{i} \}_{i\in I}$ is also a $K$-frame  for $\mathcal{H}$ with $K$-dual $\{ K^{*}S_{F}^{-1}\pi_{S_{F}(R(K))}f_{i} \}_{i\in I}$.
\end{lem}
\begin{proof}
Since the operator $S_{F}:R(K)\rightarrow S_{F}(R(K))$ is invertible. This follows that $\{ K^{*}S_{F}^{-1}\pi_{S(R(K))}f_{i} \}_{i\in I}$ is a Bessel sequence. Hence
\begin{eqnarray*}
Kf 
&=& S_{F}^{*}(S_{F}^{-1})^{*}Kf\\
&=&\pi_{R(K)}S_{F}\vert_{S_{F}(R(K))}(S_{F}^{-1})^{*}Kf \\
&=&\sum_{i\in I} \left\langle \pi_{S_{F}(R(K))}(S_{F}^{-1})^{*}Kf, f_{i}\right\rangle \pi_{R(K)}f_{i}\\
&=& \sum_{i\in I} \left\langle f, K^{*}S_{F}^{-1}\pi_{S_{F}(R(K))}f_{i}\right\rangle \pi_{R(K)}f_{i},
\end{eqnarray*}
for all $f\in \mathcal{H}$. So the result follows.
\end{proof}

\begin{thm}\label{minimal dual}
Let $K$ be a closed range operator and $W = \lbrace (W_{i}, \omega_{i})\rbrace_{i\in I}$ a minimal $K$-fusion frame for $\mathcal{H}$  with the canonical $K$-dual $\widetilde{W}$. Also, assume that $\overline{span}\{W_{i}\}_{i\in I}\cap R(K)^{\perp}=\{0\}$.  Then a Bessel fusion sequence $V = \lbrace (V_{i}, \omega_{i})\rbrace_{i\in I}$ is a $K$-dual of $W$ if and only if
\begin{eqnarray*}
 K^{*}S_{W}^{-1}\pi_{S_{W}(R(K))}W_{i} \subseteq V_{i},\quad  (i\in I).
 \end{eqnarray*}
\end{thm}
\begin{proof}
Suppose that   $\{e_{i,j}\}_{j\in J_{i}}$ is an orthonormal basis of $W_{i}$, for all $i\in I$. Then we can easily see that  the sequence $F    = \{\omega_{i}e_{i,j}\}_{i\in I, j\in J_{i}}$ is a $K$-minimal frame for $\mathcal{H}$ and $S_{F} = S_{W}$. Hence,   $\{\omega_{i}\pi_{R(K)}e_{i,j}\}_{i\in I, j\in J_{i}}$ is a $K$-minimal frame for $\mathcal{H}$, so   it has a unique $K$-dual by Theorem 6 in \cite{arefi3} and this dual is  $\{K^{*}S_{W}^{-1}\pi_{S_{W}(R(K))}\omega_{i}e_{i,j}\}_{i\in I, j\in J_{i}}$, by Lemma \ref{prop1R}. 
 Now, let $V$ be  a $K$-dual of $W$. Then
 \begin{eqnarray*}
Kf &=& \sum_{i\in I}\omega_{i}^{2}\pi_{R(K)}\pi_{W_{i}}(S_{W}^{-1})^{*}K\pi_{V_{i}}f\\
&=& \sum_{i\in I}\omega_{i}^{2}\pi_{R(K)}\sum_{j\in J_{i}}\langle \pi_{W_{i}}(S_{W}^{-1})^{*}K\pi_{V_{i}}f, e_{i,j}\rangle e_{i,j}\\
&=& \sum_{i\in I, j\in J_{i}}\langle f, \pi_{V_{i}} K^{*}S_{W}^{-1}\pi_{S_{W}(R(K))}\omega_{i}e_{i,j}\rangle \omega_{i}\pi_{R(K)}e_{i,j},
 \end{eqnarray*}
for every $f\in \mathcal{H}$.
This shows that the sequence $\{\pi_{V_{i}} K^{*}S_{W}^{-1}\pi_{S_{W}(R(K))}\omega_{i}e_{i,j}\}_{i\in I, j\in J_{i}}$ is a $K$-dual of $F$. Hence,
\begin{eqnarray*}
  \pi_{V_{i}} K^{*}S_{W}^{-1}\pi_{S_{W}(R(K))}\omega_{i}e_{i,j} = K^{*}S_{W}^{-1}\pi_{S_{W}(R(K))}\omega_{i}e_{i,j}, \quad (i\in I,  j\in J_{i}).
\end{eqnarray*}
 Thus
\begin{eqnarray*}
\pi_{V_{i}} K^{*}S_{W}^{-1}\pi_{S_{W}(R(K))}W_{i} = K^{*}S_{W}^{-1}\pi_{S_{W}(R(K))}W_{i},
 \end{eqnarray*}
 i.e.,
$K^{*}S_{W}^{-1}\pi_{S_{W}(R(K))}W_{i}\subseteq V_{i}$, for all $i\in I$. Conversely, if  a Bessel fusion sequence $V$ satisfies $K^{*}S_{W}^{-1}\pi_{S_{W}(R(K))}W_{i} \subseteq V_{i}$, for all $i\in I$. Then
 \begin{eqnarray*}
T_{V}\phi_{vw}^{*}T_{W}^{*}&=&\sum_{i\in I}\omega_{i}^{2}\pi_{V_{i}}K^{*}S_{W}^{-1}\pi_{S_{W}(R(K))}\pi_{W_{i}}\pi_{R(K)} \\
&=& K^{*}S_{W}^{-1}S_{W}\pi_{R(K)} = K^{*}.
 \end{eqnarray*}
This shows that $V$ is a $K$-dual of $W$.
\end{proof}
As a consequence we regain the following result, which was  proved in  \cite{arefi1} for fusion frames. 
\begin{cor}
Let $W = \lbrace (W_{i}, \omega_{i})\rbrace_{i\in I}$ be a minimal fusion frame for $\mathcal{H}$. Then a Bessel fusion sequence $V = \lbrace (V_{i}, \omega_{i})\rbrace_{i\in I}$ is a dual of $W$ if and only if $S_{W}^{-1}W_{i} \subseteq V_{i}$, for all $i\in I$.
\end{cor}
\begin{rem}\label{1212}
Consider a $K$-fusion frame $W = \lbrace (W_{i}, \omega_{i})\rbrace_{i\in I}$
 for $\mathcal{H}$ and let $F_{i} = \{f_{i,j}\}_{j\in J_{i}}$ be a frame for $W_{i}$, for each $i\in I$ with frame bounds $A_{i}$ and $B_{i}$, respectively such that $0<A = \inf_{i\in I}A_{i} \leq B = \sup_{i\in I}B_{i} <\infty$. Then the sequences $\{f_{i,j}\}_{j\in J_{i}}$  are called local frames of $W$ and $\{(W_{i},  \omega_{i}, \{f_{i,j}\}_{j\in J_{i}})\}_{i\in I}$ is called a $K$-fusion frame system. Also, if $\{\widetilde{f}_{i,j}\}_{j\in J_{i}}$ is a dual for $F_{i}$ in $W_{i}$, we call $\{\widetilde{f}_{i,j}\}_{j\in J_{i}}$  local dual frames.
\end{rem}
The following results describe the duality of $K$-fusion frames with
respect to local frames.
 \begin{thm}\label{3.8} Let $W = \lbrace (W_{i},
\omega_{i})\rbrace_{i\in I} $ be a $K$-fusion frame  and $V = \lbrace (V_{i},
\upsilon_{i})\rbrace_{i\in I}$ be a Bessel fusion sequence. Also,  let  $\lbrace
g_{i,j}\rbrace_{j\in J_{i}}$ be a local frame for $V_{i}$ with bounds $A_{i}$ and
$B_{i}$, for all $i\in I$ and the canonical local  dual frame $\lbrace \widetilde{g}_{i,j}\rbrace_{j\in J_{i}}$. Then $V$
is  a $K$-dual of $W$ if and only if   the sequence $G = \lbrace
\upsilon_{i}g_{i,j}\rbrace_{i\in I, j\in J_{i}}$ is a $K$-dual of $F
= \lbrace  \omega_{i} \pi_{R(K)}\pi_{W_{i}}(S_{W}^{-1})^{*}K\widetilde{g}_{i,j}\rbrace_{i\in I,
j\in J_{i}}$.\end{thm}
\begin{proof}  We first show that $F$ and $G$ are  Bessel sequences for $\mathcal{H}$.
 \begin{eqnarray*} &&\sum_{i\in I, j\in
J_{i}} \vert \langle f,
\omega_{i}\pi_{R(K)}\pi_{W_{i}}(S_{W}^{-1})^{*}K\widetilde{g}_{i,j} \rangle\vert^{2} \\
&=&
\sum_{i\in I, j\in J_{i}} \vert \langle
\omega_{i}K^{*}S_{W}^{-1}\pi_{S_{W}(R(K))}\pi_{W_{i}}\pi_{R(K)}f, \widetilde{g}_{i,j}\rangle\vert^{2}\\
&=&\sum_{i\in I}\omega_{i}^{2}\sum_{j\in J_{i}} \vert \langle
\pi_{V_{i}}K^{*}S_{W}^{-1}\pi_{S_{W}(R(K))}\pi_{W_{i}}\pi_{R(K)}f, \widetilde{g}_{i,j}\rangle\vert^{2}\\
&\leq& \sum_{i\in I} \dfrac{\omega_{i}^{2}}{A_{i}} \Vert
\pi_{V_{i}}K^{*}S_{W}^{-1}\pi_{S_{W}(R(K))}\pi_{W_{i}}\pi_{R(K)}f\Vert^{2}\\
&\leq& \dfrac{\Vert K\Vert^{2}\Vert S_{W}^{-1}\pi_{S_{W}(R(K))}\Vert^{2}}{A} \sum_{i\in I} \omega_{i}^{2}
\Vert \pi_{W_{i}}\pi_{R(K)}f\Vert^{2}\\
&\leq&  \dfrac{\Vert S_{W}^{-1}\pi_{S_{W}(R(K))}\Vert^{2}}{A} D\Vert K\Vert^{2} \Vert f\Vert^{2},\end{eqnarray*}
for every $f\in \mathcal{H}$, where $D$ is an upper bound for $W$ and $A = \inf_{i\in I}A_{i}$. Moreover, we have  \begin{eqnarray*}
\sum_{i\in I, j\in
J_{i}} \vert \langle f,
\upsilon_{i}g_{i,j}\rangle\vert^{2} &=& \sum_{i\in I}\upsilon_{i}^{2} \sum_{ j\in
J_{i}}\vert \langle \pi_{V_{i}}f,
g_{i,j}\rangle\vert^{2} \\
&\leq& \sum_{i\in I} B_{i} \upsilon_{i}^{2} \Vert \pi_{V_{i}}f\Vert^{2}\\
& \leq& B \sum_{i\in I}\upsilon_{i}^{2} \Vert \pi_{V_{i}}f\Vert^{2},
\end{eqnarray*}
where $ B = \sup_{i\in I}B_{i}$.
 On the other hand,
\begin{eqnarray*} T_{G}T_{F}^{*}f &=&   \sum_{i\in I, j\in J_{i}} \langle f, \upsilon_{i}g_{i,j}\rangle  \omega_{i} \pi_{R(K)}\pi_{W_{i}}(S_{W}^{-1})^{*}K\widetilde{g}_{i,j}\\
&=& \sum_{i\in I}
\omega_{i}\upsilon_{i} \pi_{R(K)}\pi_{W_{i}}(S_{W}^{-1})^{*}K \sum_{j\in J_{i}} \langle
\pi_{V_{i}}f,
g_{i,j}\rangle \widetilde{g}_{i,j}\\
&=&
 \sum_{i\in I}\omega_{i}\upsilon_{i} \pi_{R(K)}\pi_{W_{i}}(S_{W}^{-1})^{*}K\pi_{V_{i}}f = \pi_{R(K)}T_{W}\phi_{vw}T^{*}_{V}f.\end{eqnarray*}
Hence, $V=\lbrace(V_{i},\upsilon_{i})\rbrace_{i\in I}$ is a $K$-dual of
$W$ if and only if  $G$ is a $K$-dual of $F$.\end{proof}
By a similar argument to the proof of Theorem \ref{3.8} one may prove the next theorem.

\begin{thm}\label{local2}
Let $W = \lbrace (W_{i}, \omega_{i})\rbrace_{i\in I}$ be a $K$-fusion frame
with  bounds $A$ and $B$, respectively. A Bessel fusion sequence
$V=\lbrace(V_{i},\upsilon_{i})\rbrace_{i\in I}$ is a $K$-dual of $W$ if and
only if $G=\lbrace \upsilon_{i}\pi_{V_{i}}e_{j}\rbrace_{i\in I, j\in J}$ is
a $K$-dual of $F=\lbrace \omega_{i}
\pi_{R(K)}\pi_{W_{i}}(S_{W}^{-1})^{*}Ke_{j}\rbrace_{i\in I, j\in J}$,
where $\lbrace e_{j}\rbrace_{j\in J}$ is an orthonormal basis of
$\mathcal{H}$.
\end{thm}

The following result shows that for every  local frame of a $K$-fusion frame we can construct some  $K$-frames with associated $K$-duals.
\begin{prop}
Let $W = \lbrace (W_{i}, \omega_{i})\rbrace_{i\in I}$ be a
$K$-fusion frame for $\mathcal{H}$ and $\{f_{i,j}\}_{j\in J_{i}}$ be a local frame for $W_{i}$ with the local dual frame $\{\widetilde{f}_{i,j}\}_{j\in J_{i}}$, for all $i\in I$. Then $\{\omega_{i}f_{i,j}\}_{i\in I, j\in J_{i}}$ is a $K$-frame for $\mathcal{H}$ with $K$-dual $G = \{X_{i}^{*}\widetilde{f}_{i,j}\}_{i\in I, j\in J_{i}}$, where the operator $X$ is as in $(\ref{DougX})$.
\end{prop}
\begin{proof}
  First, note that $G = \{X_{i}^{*}\widetilde{f}_{i,j}\}_{i\in I, j\in J_{i}}$ is a Bessel sequence. In fact,
   \begin{eqnarray*}
\sum_{i\in I, j\in J_{i}}\vert \langle f, X_{i}^{*}\widetilde{f}_{i,j}\rangle\vert^{2}&=& \sum_{i\in I} \sum_{j\in J_{i}}\vert \langle \pi_{W_{i}}X_{i}f, \widetilde{f}_{i,j}\rangle\vert^{2}\\
&\leq& B  \sum_{i\in I}\Vert \pi_{W_{i}}X_{i}f\Vert^{2}\\
&\leq& B  \Vert X\Vert^{2}\Vert f\Vert^{2},
 \end{eqnarray*}
for every $f\in \mathcal{H}$  where $B$  is given by Remark \ref{1212}. Also, similar to Theorem 3.2 of \cite{Cas04}, we can see that $F = \{\omega_{i}f_{i,j}\}_{i\in I, j\in J_{i}}$ is a $K$-frame for $\mathcal{H}$. Moreover,
  \begin{eqnarray*}
\sum_{i\in I, j\in J_{i}} \langle f, X_{i}^{*}\widetilde{f}_{i,j}\rangle \omega_{i}f_{i,j}
&=& \sum_{i\in I}\omega_{i}\sum_{j\in J_{i}} \langle \pi_{W_{i}}X_{i}f, \widetilde{f}_{i,j}\rangle f_{i,j}\\
&=&  \sum_{i\in I}\omega_{i} \pi_{W_{i}}X_{i}f \\
&=& T_{W}Xf  = Kf.
 \end{eqnarray*}
Hence $G$ is a $K^{*}$-frame and also a
  $K$-dual of  $F$.
\end{proof}

\section{Resolution of  bounded linear operators}

\smallskip
\goodbreak

The concept of  resolution of the identity has been considered in \cite{ Cas04, khosravi15}. In this section,  we introduce the notion of resolution of a bounded linear operator $K\in B(\mathcal{H})$, which lead to more reconstructions from the elements of  $R(K)$.
\smallskip
\goodbreak
Let $K\in B(\mathcal{H})$ and $\{\theta_{i}\}_{i\in I}$ be  a family of bounded linear operators  on $\mathcal{H}$, we say $\{\theta_{i}\}_{i\in I}$ is an  $l^{2}$-resolution of  $K$  with  respect to  a family of weights $\{\omega_{i}\}_{i\in I}$ for $\mathcal{H}$   whenever there exists a positive constant $B$ such that
\begin{itemize}
\item[(i)] $Kf = \sum_{i\in I}\omega_{i}^{2}\theta_{i}f$,
\item[(ii)] $\sum_{i\in I}\omega_i^2\|\theta_{i}f\|^2\leq B\|f\|^{2}$,
\end{itemize}
for every $f\in \mathcal{H}$.
If $\theta := \{\theta_{i}\}_{i\in I}$ only  satisfies $(i)$ we say $\theta$ is a resolution of the operator $K$.
\begin{rem}\label{cons. l-reso1}
\begin{itemize}

\item[(1)] One can easily shows that for every $l^{2}$-resolution $\{\theta_{i}\}_{i\in I}$ of an operator  $K\in  B(\mathcal{H})$ there exists $A>0$ such that
\begin{eqnarray*}
A\|Kf\|^{2}\leq \sum_{i\in I}\omega_i^2\|\theta_{i}f\|^2
\end{eqnarray*}
\item[(2)] Let  $W=\lbrace
(W_{i},\omega_{i})\rbrace_{i\in I}$ be a $K$-fusion frame  for $\mathcal{H}$. Then

\item[a)]   There exists a bounded operator $X\in B(\mathcal{H},  \sum_{i\in
I}\oplus W_{i} )$, such that $K = T_{W}X$ by Theorem \ref{equ0}. Hence, the operators   $\theta_{i}: \mathcal{H}\rightarrow W_{i}$  given by  $\theta_{i}f = X_{i}f$, where $X_{i}f$ is the $i$-th component of $Xf$, constitute an $l^{2}$-resolution of  $K$ with  respect to  the family of weights $\{\sqrt{\omega_{i}}\}_{i\in I}$.
\item[b)] Define $\theta_{i}\in B(\mathcal{H})$ by $\theta_{i} = \pi_{R(K)}\pi_{W_{i}}(S_{W}^{-1})^{*}K$, for all $i\in I$. Then $\{\theta_{i}\}_{i\in I}$ is an  $l^{2}$-resolution of  $K$ with  respect to $\{\omega_{i}\}_{i\in I}$.
\item[c)]Suppose $\theta_{i}\in B(\mathcal{H})$ is given by $\theta_{i} = S_{W}^{-1}\pi_{S_{W}(R(K))}\pi_{W_{i}}K$, for all $i\in I$. Then $\{\theta_{i}\}_{i\in I}$   is an  $l^{2}$-resolution of  $K$ with  respect to   $\{\omega_{i}\}_{i\in I}$.
\end{itemize}
\end{rem}
As we observed,  by using $K$-fusion frames we will obtain many resolutions  of the operator $K$. In the following  proposition, we show  that by an  $l^{2}$-resolution of  the operator $K$ one may construct a $K$-fusion frame.
\begin{prop}
Let $\{\theta_{i}\}_{i\in I}$ be an $l^{2}$-resolution of  $K$ with respect to $\{\omega_{i}\}_{i\in I}$ for $\mathcal{H}$, such that   $W=\{(\overline{R(\theta_{i})}, \omega_{i})\}_{i\in I}$ constitute a Bessel fusion sequence. Then $W$ is a  $K$-fusion frame for $\mathcal{H}$.
\end{prop}
\begin{proof}
By using assumption, it is enough to show the existence of a lower bound for $W$,
\begin{eqnarray*}
\Vert K^{*}f\Vert^{4} &=& \vert \langle K^{*}f, K^{*}f\rangle\vert^{2}\\
&=& \left( \sum_{i\in I}\omega_{i}^{2}\langle \pi_{\overline{R(\theta_{i})}}f, \theta_{i}K^{*}f\rangle \right)^{2}\\
&\leq& \sum_{i\in I}\omega_{i}^{2}\Vert \pi_{\overline{R(\theta_{i})}}f\Vert^{2}  \sum_{i\in I}\omega_{i}^{2}\Vert \theta_{i}K^{*}f\Vert^{2} \\
&\leq& B\Vert K^{*}f\Vert^{2}\sum_{i\in I}\omega_{i}^{2}\Vert \pi_{\overline{R(\theta_{i})}}f\Vert^{2},
\end{eqnarray*}
 for every $f\in \mathcal{H}$, in which $B$ is an upper bound of $\{\theta_{i}\}_{i\in I}$. Thus, the result follows.
\end{proof}
The next theorem shows that the $l^{2}$-resolution constructed by $X_{w}$  has minimum $l^{2}$-norm between all $l^{2}$-resolutions of the operator $K$, where $X_{w}$ is as in $(\ref{XW})$.
\begin{thm}\label{minimal norm}
Let $W=\lbrace
(W_{i},\omega_{i})\rbrace_{i\in I}$ be a $K$-fusion frame for $\mathcal{H}$ and the operators $\theta_{i}: \mathcal{H}\rightarrow W_{i}$ constitute an $l^{2}$-resolution of  $K$. Then
\begin{eqnarray*}
\sum_{i\in I}\|\pi_{W_{i}}(X_{w}f)_{i\in I}\|^2\leq  \sum_{i\in
I}\|\theta_{i}f\|^2, \quad (f\in \mathcal{H}).
\end{eqnarray*}
Furthermore,
\begin{eqnarray*}
\sum_{i\in I}\|\pi_{W_{i}}(X_{w}f)_{i\in I} - \omega_{i}\pi_{W_i}f\|^2\leq  \sum_{i\in
I}\|\theta_{i}f - \omega_{i}\pi_{W_i}f\|^2, \quad (f\in \mathcal{H}).
\end{eqnarray*}
\end{thm}
\begin{proof}
Suppose that $\{\theta_{i}\}_{i\in I}$ is an  $l^{2}$-resolution of $K$. Define  $\theta: \mathcal{H}\rightarrow \sum_{i\in I}\oplus W_{i}$ by $\theta f=\lbrace \theta_{i}f\rbrace_{i\in I}$. Then $\theta$ is a bounded linear operator and $T_{W}\theta f= \sum_{i\in I}\theta_{i}f=Kf$. Moreover,
\begin{eqnarray*}
\Vert K^{*}f\Vert^{2} \leq \Vert \theta\Vert^{2}\Vert T_{W}^{*}f\Vert^{2},
\end{eqnarray*}
which follows that
\begin{eqnarray*}
\Vert X_{w}\Vert ^{2} = \inf \{\alpha>0, \Vert K^{*}f\Vert ^{2}\leq \alpha\Vert T_{W}^{*}f\Vert ^{2}; f\in \mathcal{H}\},
\end{eqnarray*}
as required. On the other hand
\begin{eqnarray*}
\sum_{i\in I}\|\pi_{W_i}(X_{w}f)_{i}f-\omega_{i}\pi_{W_{i}}f\|^2 &=& \sum_{i\in I}\|\pi_{W_i}(X_{w}f)_{i}\|^2\\
&+&\sum_{i\in I}\omega_{i}^{2}\|\pi_{W_{i}}f\|^2 - 2 Re\sum_{i\in I}\langle \omega_{i}\pi_{W_i}(X_{w}f)_{i}f, f\rangle\\
&\leq&  \sum_{i\in
I}\|\theta_{i}f\|^2 + \sum_{i\in I}\omega_{i}^{2}\|\pi_{W_{i}}f\|^2 \\
&-& 2  Re\langle Kf, f\rangle
= \sum_{i\in I}\|\theta_{i}f-\omega_{i}\pi_{W_{i}}f\|^2.
\end{eqnarray*}
This completes the proof.
\end{proof}
 In the  case $K=I_{\mathcal{H}}$, the above theorem reduces
to a result in \cite{khosravi15}. As a result of  Theorem \ref{minimal norm}, we can obtain  the pseudo-inverse of  the bounded operator $\pi_{R(K)}T_{W}$.
\begin{cor}
Suppose that $\lbrace
(W_{i},\omega_{i})\rbrace_{i\in I}$ is a $K$-fusion frame for $\mathcal{H}$. Then the pseudo-inverse operator $(\pi_{R(K)}T_{W})^{\dagger}: R(K)\rightarrow \sum_{i\in I}\oplus W_{i}$ is given by
\begin{eqnarray*}
(\pi_{R(K)}T_{W})^{\dagger}f = \{\omega_{i}\pi_{W_i}(X_{w}f)_{i}\}_{i\in I}, \quad (f\in R(K)).
\end{eqnarray*}
\end{cor}
\begin{proof}
For a $K$-fusion frame $W = \lbrace
(W_{i},\omega_{i})\rbrace_{i\in I}$ we can easily survey that $T_{W}^{*}\vert_{R(K)}$ is a one to one operator, so the operator $\pi_{R(K)}T_{W}:  \sum_{i\in I}\oplus W_{i} \rightarrow R(K)$ is onto. Let $f\in \mathcal{H}$, by Corollary 1.1  in \cite{Berut76},  the equation
$Kf = \pi_{R(K)}T_{W}\{f_{i}\}_{i\in I}$ has a unique solution of  minimal norm and this solution is $(\pi_{R(K)}T_{W})^{\dagger}Kf$. On the other hand,
\begin{eqnarray*}
\pi_{R(K)}T_{W} \{\pi_{W_i}(X_{w}f)_{i}\}_{i\in I} 
= \pi_{R(K)}T_{W}X_{w}f =Kf.
\end{eqnarray*}
Thus,  the result follows  by Theorem \ref{minimal norm}.
\end{proof}

\section{Perturbation of $K$-fusion frames}

\smallskip
\goodbreak
In fusion  frame theory, the elements of underlying Hilbert spaces are distributed to a family of closed subspaces. These elements can be reconstructed by dual fusion frames  such as (\ref{Def:alt}).
 In real
applications, under these transmissions usually a part of the data vectors
change or reshape, in the other words, the various disturbances and perturbations  affect on the information.
In this respect,  stability of fusion frames and dual fusion frames  under perturbations has a key role in
practice.  
In this
section, we study the  robustness of $K$-fusion frames and their $K$-duals under some perturbations.
 \begin{thm}\label{per1} Let $W=\lbrace
(W_{i},\omega_{i})\rbrace_{i\in I}$ be a $K$-fusion frame for $\mathcal{H}$
with bounds $A$ and $B$, respectively. Also, let $Z=\lbrace (Z_{i},z_{i})\rbrace_{i\in I}$ be a $(\lambda_{1},
\lambda_{2},\varepsilon)$-perturbation of $W$  for some
 $0 <
\lambda_{1}, \lambda_{2} < 1$ and $\varepsilon >0$, i.e.,
\begin{eqnarray*} \Vert
(\omega_{i}\pi_{W_{i}}-z_{i}\pi_{Z_{i}})f\Vert \leq \lambda_{1}\Vert
\omega_{i}\pi_{W_{i}}f\Vert+\lambda_{2}\Vert
z_{i}\pi_{Z_{i}}\Vert+\varepsilon \omega_{i}\Vert K^{*}f\Vert,
\end{eqnarray*} for all $i\in I$ and $f\in \mathcal{H}$
 such that
\begin{eqnarray}\label{varepsilon}
\varepsilon < \frac{(1-\lambda_{1})\sqrt{A}}{\Vert K\Vert(\sum_{i\in
I}\omega_{i}^{2})^{1/2}}
\end{eqnarray}
Then $Z$ is a $K$-fusion frame
for $\mathcal{H}$. \end{thm}
\begin{proof}
We first show the existence of a Bessel bound for $Z$. Let  $f\in \mathcal{H}$, 
\begin{eqnarray*}
&&\left(\sum_{i\in I}
z_{i}^{2}\Vert \pi_{Z_{i}}f\Vert^{2}\right)^{1/2}\\
&=&\left(\sum_{i\in I}
\Vert z_{i}\pi_{Z_{i}}f+\omega_{i}\pi_{W_{i}}f-\omega_{i}\pi_{W_{i}}f\Vert^{2}\right)^{1/2}\\
&\leq& \left(\sum_{i\in I} \left( \omega_{i}\Vert
\pi_{W_{i}}f\Vert+\lambda_{1}\Vert
\omega_{i}\pi_{W_{i}}f\Vert+\lambda_{2}\Vert
z_{i}\pi_{Z_{i}}f\Vert+\varepsilon \omega_{i}\Vert
K^{*}f\Vert \right)^{2}\right)^{1/2}\\
&\leq& (1+\lambda_{1})\left(\sum_{i\in
I} \omega_{i}^{2}\Vert
\pi_{W_{i}}f\Vert^{2}\right)^{1/2}+\lambda_{2}\left(\sum_{i\in I} z_{i}^{2}\Vert \pi_{Z_{i}}f\Vert^{2}\right)^{1/2}\\
&+& \left(\sum_{i\in I}\omega_{i}^{2}\right)^{1/2}\varepsilon \Vert K^{*}f\Vert,  \end{eqnarray*}
for every $f\in \mathcal{H}$. Using (\ref{varepsilon}), shows that \begin{eqnarray*} \sum_{i\in I} z_{i}^{2}\Vert \pi_{Z_{i}}f\Vert^{2}
\leq \left(\dfrac{(1+\lambda_{1})\sqrt{B}+\varepsilon \Vert K\Vert
\left(\sum_{i\in I}\omega_{i}^{2}\right)^{1/2}}{1-\lambda_{2}}\right)^{2}
\Vert f\Vert^{2}.
\end{eqnarray*}
Now, it is sufficient to find a lower bound. In fact,
\begin{eqnarray*} &&\left(\sum_{i\in I} z_{i}^{2}\left\Vert
\pi_{Z_{i}}f\right\Vert^{2}\right)^{1/2}\\
&=& \left(\sum_{i\in I} \left\Vert
z_{i}\pi_{Z_{i}}f+\omega_{i}\pi_{W_{i}}f-\omega_{i}\pi_{W_{i}}f\right\Vert^{2}\right)^{1/2}
\\ &\geq& \left(\sum_{i\in I} \left(\omega_{i}\Vert
\pi_{W_{i}}f\Vert-\lambda_{1}\Vert \omega_{i}\pi_{W_{i}}f-\lambda_{2}\Vert
z_{i}\pi_{Z_{i}}\Vert-\varepsilon \omega_{i}\Vert
K^{*}f\Vert\right)^{2}\right)^{1/2}\\ &\geq& (1-\lambda_{1})\left(\sum_{i\in
I} \omega_{i}^{2}\Vert
\pi_{W_{i}}f\Vert^{2}\right)^{1/2}-\lambda_{2}\left(\sum_{i\in I}
z_{i}^{2}\Vert \pi_{Z_{i}}f\Vert^{2}\right)^{1/2}\\ &-& \left(\sum_{i\in
I}\omega_{i}^{2}\right)^{1/2}\varepsilon \Vert K^{*}f\Vert. \end{eqnarray*}
Therefore  \begin{eqnarray*} \sum_{i\in I} z_{i}^{2}\Vert
\pi_{Z_{i}}f\Vert^{2} \geq \left(\dfrac{(1-\lambda_{1})\sqrt{A}-\varepsilon
\Vert K\Vert \left(\sum_{i\in
I}\omega_{i}^{2}\right)^{1/2}}{1+\lambda_{2}}\right)^{2} \Vert
K^{*}f\Vert^{2}.
\end{eqnarray*} This completes the proof. \end{proof} Take $\omega_{i}=z_{i}$
for all $i\in I$ and $K=I_{\mathcal{H}}$, then Theorem \ref{per1} reduces  in
Proposition 5.2 of \cite{Cas08}. In the next theorem we show that under  small perturbations, $K$-duals of a $K$-fusion frame turn to the  approximate $K$-dual for perturbed $K$-fusion frame. Let  $W=\lbrace
(W_{i},\omega_{i})\rbrace_{i\in I}$ be a $K$-fusion frame for $\mathcal{H}$.  A Bessel fusion sequence $V=\lbrace(V_{i},\nu_{i})\rbrace_{i\in I}$ is called  an approximate $K$-dual of  $W$ whenever $\Vert K -  T_{W}\phi_{vw}T^{*}_{V}\  \Vert<1$.
 Approximate duals  was first  introduced in \cite{app.} for discrete frames and
are  important tools for reconstruction algorithms.

\begin{thm}\label{last} Let $W=\lbrace
(W_{i},\omega_{i})\rbrace_{i\in I}$ be a $K$-fusion frame for $\mathcal{H}$
with bounds $A$ and $B$, respectively. Also, let $\{Z_{i}\}_{i\in I}$ be a family of closed subspaces in $\mathcal{H}$.
 \begin{eqnarray*} \Vert
(T^{*}_{W}-T^{*}_{Z})f\Vert \leq \varepsilon \Vert K^{*}f\Vert, \end{eqnarray*}
for some   $\varepsilon> 0$. Then,
\begin{itemize}
\item[(i)]
If $0 < \varepsilon < \sqrt{A}$, then $Z = \lbrace
(Z_{i},\omega_{i})\rbrace_{i\in I}$ is a $K$-fusion frame for $\mathcal{H}$ with the bounds $(\sqrt{A}-\varepsilon)$ and
$(\sqrt{B} + \varepsilon \Vert K\Vert)$, respectively.
\item[(ii)]
Every $K$-dual $V=\lbrace(V_{i},\omega_{i})\rbrace_{i\in I}$ of $W$ is an approximate $K$-dual of  $Z$, for  $\varepsilon> 0$  sufficiently small, .
\end{itemize}
 \end{thm}
\begin{proof} For every $f\in \mathcal{H}$, we can write \begin{eqnarray*}
(\sqrt{A}-\varepsilon) \Vert K^{*}f\Vert &\leq& \Vert T^{*}_{W}f\Vert -\Vert
T^{*}_{Z}f-T^{*}_{W}f\Vert\\ &\leq& \Vert T^{*}_{Z}f\Vert \leq \Vert
T^{*}_{W}f\Vert +\Vert T^{*}_{Z}f-T^{*}_{W}f\Vert
\\ &\leq& (\sqrt{B} + \varepsilon \Vert K\Vert)\Vert f\Vert, \end{eqnarray*}
This shows $(i)$. Moreover, let  $V=\lbrace(V_{i},\omega_{i})\rbrace_{i\in I}$ be a $K$-dual  of $W$. Then
\begin{eqnarray*}
&& \left\Vert K^{*}f - T_{V}\phi_{vz}^{*}T_{Z}^{*}\pi_{R(K)}f\right\Vert^{2}\\
&=& \left\Vert K^{*}f -\sum_{i\in I}\omega_{i}^{2}\pi_{V_{i}}K^{*}S_{Z}^{-1}\pi_{S_{Z}(R(K))}\pi_{Z_{i}}\pi_{R(K)}f
\right \Vert^{2}\\
&=&\left \Vert \sum_{i\in I}\omega_{i}^{2}\pi_{V_{i}}K^{*}S_{W}^{-1}\pi_{S_{W}(R(K))}\pi_{W_{i}}f
-\sum_{i\in I}\omega_{i}^{2}\pi_{V_{i}}K^{*}S_{Z}^{-1}\pi_{S_{Z}(R(K))}\pi_{Z_{i}}\pi_{R(K)}f
\right \Vert^{2}\\
&\leq&  2\left \Vert \sum_{i\in I}\omega_{i}^{2}\pi_{V_{i}}K^{*}(S_{W}^{-1}\pi_{S_{W}(R(K))}-S_{Z}^{-1}\pi_{S_{Z}(R(K))})\pi_{W_{i}}\pi_{R(K)}f\right \Vert^{2}\\
&+&2\left\Vert \sum_{i\in I}\omega_{i}^{2}\pi_{V_{i}}K^{*}S_{Z}^{-1}\pi_{S_{Z}(R(K))}(\pi_{W_{i}}-\pi_{Z_{i}})\pi_{R(K)}f
\right \Vert^{2}\\
&\leq& 2\left( \left\Vert \left((S_{W}^{-1})^{*}-(S_{Z}^{-1})^{*}\right)K\right\Vert^{2}B\left\Vert f \right\Vert^{2} + \varepsilon\left\Vert (S_{Z}^{-1})^{*}K\right\Vert^{2}\left\Vert K\right\Vert^{2}\left\Vert f\right\Vert^{2}
\right),
\end{eqnarray*}
Now, suppose that
\begin{eqnarray*}
0 < \varepsilon < \min \left( \sqrt{A}, \dfrac{\dfrac{1}{2} -\left\Vert \left((S_{W}^{-1})^{*}-(S_{Z}^{-1})^{*}\right)K\right\Vert^{2}B}{\Vert (S_{Z}^{-1})^{*}K\Vert^{2}\Vert K\Vert^{2}} \right).
\end{eqnarray*}
Therefore,
\begin{eqnarray*}
\Vert K^{*}f - T_{V}\phi_{vz}^{*}T_{Z}^{*}\pi_{R(K)}f\Vert< \Vert f\Vert,
\end{eqnarray*}
for every $f\in \mathcal{H}$.  This implies that  $\Vert K -  T_{Z}\phi_{vz}T^{*}_{V}\pi_{R(K)}\Vert< 1$, as required.
\end{proof}
\begin{ex}
Suppose that $\mathcal{H}$, $K$ and  $W=\lbrace (W_{i},
\omega_{i})\rbrace_{i=1}^{3}$  are as in Example \ref{example2}. Also, let
\begin{eqnarray*}
Z_{1}=W_{1}, \quad Z_{2}=W_{2}\oplus W_{3}, \quad Z_{3}=W_{3}.
\end{eqnarray*}
Then, the Bessel fusion  sequence $Z  = \{(Z_{i},1) \}_{i=1}^{3}$ satisfies
\begin{eqnarray*}
\Vert
(T^{*}_{W}-T^{*}_{Z})f\Vert   < 1/2  \Vert K^{*}f\Vert,
\end{eqnarray*}
for every $f=(a,b, c)\in \mathcal{H}$, i.e., $Z$ is an $\varepsilon$-perturbation of $W$ with $\varepsilon  = 1/2$. Thus $Z$ is a $K$-fusion frame by Theorem \ref{last} (i). Now, a direct calculation shows that
\begin{equation*}
S_{Z}\pi_{R(K)} = \left[
 \begin{array}{ccc}

3/2 \quad 3/2 \quad 0\\

3/2 \quad 3/2 \quad 0\\

0 \quad 0 \quad 2\\

\end{array} \right],
\end{equation*}
and consequently
\begin{equation*}
S_{Z}^{-1}\pi_{S_{Z}(R(K))} = \left[
 \begin{array}{ccc}

1/6 \quad 1/6 \quad 0\\

1/6 \quad 1/6 \quad 0\\

0 \quad 0 \quad 1/2\\

\end{array} \right].
\end{equation*}
Hence, $\Vert K^{*}(S_{Z}^{-1}-S_{W}^{-1})\Vert =1/6$ and $\Vert (S_{Z}^{-1})^{*}K\Vert = 1/3$. Also, we have $A=1$, $B=2$ and $\Vert K\Vert=1$. Thus
\begin{eqnarray*}
\min \left( \sqrt{A}, \dfrac{\dfrac{1}{2} -\left\Vert \left((S_{W}^{-1})^{*}-(S_{Z}^{-1})^{*}\right)K\right\Vert^{2}B}{\Vert (S_{Z}^{-1})^{*}K\Vert^{2}\Vert K\Vert^{2}} \right) = 1.
\end{eqnarray*}
Therefore,  every $K$-dual  $V=\lbrace (V_{i},
1)\rbrace_{i=1}^{3}$  of $W$ is an approximate $K$-dual of $Z$, moreover $\Vert K -  T_{Z}\phi_{vz}T^{*}_{V}\pi_{R(K)} \Vert < 1/4$.
\end{ex}



\bibliographystyle{amsplain}

\end{document}